\documentclass[12pt,reqno]{amsart}
\usepackage{amsmath,amsthm,amssymb,amscd,url,enumerate}
\usepackage{graphicx}
\usepackage{afterpage}
\usepackage{tikz}
\usepackage[all]{xy}
\usepackage[colorlinks=true]{hyperref}
\usepackage[width=5.3in,height=8.5in, centering]{geometry}
\usepackage{algorithm}

\theoremstyle{theorem}
\newtheorem{theorem}{Theorem}
\newtheorem{conjecture}[theorem]{Conjecture}
\newtheorem{problem}[theorem]{Problem}
\newtheorem{proposition}[theorem]{Proposition}

\theoremstyle{definition}
\newtheorem{definition}[theorem]{Definition}

\theoremstyle{remark}
\newtheorem{remark}[theorem]{Remark}
\newtheorem{example}[theorem]{Example}
\newtheorem{question}[theorem]{Question}

\numberwithin{theorem}{section}

  {\end{list}}

  {\end{list}}

\newcommand{\tightoverset}[2]{%
  \mathop{#2}\limits^{\vbox to -.5ex{\kern-1.05ex\hbox{$#1$}\vss}}}

\def\Gcal{{\mathcal G}}

\newcommand{\CC}{\mathbb{C}}
\newcommand{\FF}{\mathbb{F}}

\newcommand{\QQ}{\mathbb{Q}}
\newcommand{\RR}{\mathbb{R}}
\newcommand{\ZZ}{\mathbb{Z}}

\newcommand{\MOD}[1]{~(\textup{mod}~#1)}
\renewcommand{\pmod}{\MOD}

\title{Ring-LWE Cryptography for the Number Theorist}

\author{Yara Elias, Kristin E. Lauter,  Ekin Ozman, Katherine E. Stange }
\address{Yara Elias:
Department of Mathematics and Statistics, McGill University, Montreal, Quebec, Canada}
\email{yara.elias@mail.mcgill.ca}
\address{Kristin E. Lauter:Microsoft Research, One Microsoft Way, Redmond, WA 98052}
\email{klauter@microsoft.com}
\address{Ekin Ozman: Department of Mathematics, Faculty of Arts and Science, Bogazici University, 34342, Bebek-Istanbul, Turkey}
\email{ekin.ozman@boun.edu.tr}
\address{Katherine E. Stange: Department of Mathematics, University of Colorado, Campux Box 395, Boulder, Colorado 80309-0395}
\email{kstange@math.colorado.edu}

\date\today

\begin{document}

\maketitle

\begin{abstract}
        In this paper, we survey the status of attacks on the \emph{ring and polynomial learning with errors problems} (RLWE and PLWE).  Recent work on the security of these problems \cite{EHL,ELOS} gives rise to interesting questions about number fields.  We extend these attacks and survey related open problems in number theory, including spectral distortion of an algebraic number and its relationship to Mahler measure, the monogenic property for the ring of integers of a number field, and the size of elements of small order modulo $q$.
\end{abstract}


\section{Introduction}



Public key cryptography relies on the existence of hard computational problems in
mathematics:  i.e., problems for which there are no known general polynomial-time algorithms.
Hard mathematical problems related to lattices were first suggested as the basis for cryptography almost two decades ago (\cite{A,AD,HPS}).
While other better-known problems in public key cryptography such as factoring and the discrete logarithm problem are closely tied to computational number theory, lattice-based cryptography has seemed somewhat more distant. Recent developments, including the introduction of the ring-learning with errors problem instantiated in the ring of integers of a number field (\cite{LPR10}), have connected the area to new questions in computational number theory.

At the same time, lattice-based cryptography has seen a dramatic surge of activity.  Since there are no known polynomial time algorithms for attacking standard lattice problems on a quantum computer (in contrast to the case for widely deployed cryptographic systems such as RSA, discrete log, and elliptic curves), lattice-based cryptography is considered to be a promising cryptographic solution in a post-quantum world.

One of the most exciting recent developments has been the construction of \emph{fully homomorphic encryption} schemes (\cite{Gen09, BV11, BV11b, BGV11,GHS11})
which allow meaningful operations to be performed on data without decrypting it:  one can add and multiply encrypted numbers, returning the encrypted correct result, \emph{without knowledge of the plaintext or private key}.  The addition and multiplication of ciphertexts is possible due to the ring structure inherent in polynomial rings:  these translate into {\tt AND} and {\tt OR} gates which can be used to build arbitrary circuits.  Exciting applications include privacy problems
in the health sector for electronic medical records, predictive analysis and learning from sensitive private data, and genomic computations (\cite{LNV,GLN,BLN,LLN}).

These new homomorphic encryption solutions are based on versions of hard ``learning problems'' with security reductions to and from standard lattice problems such as the shortest vector problem (\cite{Regev}). The idea behind the whole class
of learning problems is that it is hard to ``learn'' a secret vector, given only sample inner products of that vector with other random vectors, provided these products are obscured by adding a small amount of Gaussian noise (``errors'').

The ring version, which we call Ring-LWE or RLWE, was introduced in \cite{LPR10},
presenting a fundamental hardness result which can be described informally as follows: for any ring of integers $R$, any algorithm that solves the (search version of the) Ring-LWE problem yields a comparably efficient {\it quantum} algorithm that finds approximately shortest vectors in {\it any} ideal lattice in R.


Soon after the introduction of Ring-LWE, an efficient cryptosystem allowing for homomorphic multiplication was proposed in~\cite{BV11} based on a variant of the RLWE problem, the Polynomial Learning With Errors problem (here denoted PLWE). Improvements to that cryptosystem (e.g.~\cite{BGV11,GHS11}) have followed in the same vein, with the same hardness assumption.  The reader should note that the terminology of ``Ring-LWE'' vs. ``Poly-LWE'' is not entirely standard, and some authors use ``Ring-LWE'' to refer to a larger class of problems including both.

We focus in this paper on PLWE, specified by the following choices:

\begin{enumerate}
        \item a polynomial ring $P_q=\FF_q[x]/(f(x))$, with $f(x)$  a monic irreducible polynomial of degree $n$ over $\ZZ$ which splits completely over $\FF_q$,
        \item a basis for the polynomial ring, which will often be taken to be a power basis in the monogenic case (in particular, the choice of a basis can be used to endow the ring with the standard inner product on the ring),
        \item and a parameter specifying the size of the Gaussian noise to be added (the size of the ``error''), spherical with respect to this inner product.
\end{enumerate}

We also focus on RLWE obtained from the same setup, but with the inner product instead given by the Minkowski embedding of a ring of integers of the form $\ZZ[x]/(f(x))$.  More general situations, including the case where the defining polynomial for the number ring does not split modulo $q$, or the case where $q$ is composite, or the distribution is non-spherical or non-Gaussian, are considered in the cryptographic literature, but the setup above will suffice for our present purpose, which is to give a number theorist an entr\'ee into the subject.

A key point is that for cryptographic applications, the errors must be chosen to be relatively small, to allow for correct decryption.  For PLWE, ``small'' refers to the coefficient size (absolute value of the smallest residue), where the error is a polynomial, i.e. represented according to a polynomial basis for the ring.  
But to relate RLWE to other standard lattice problems, \cite{LPR10} considers the embedding of the ring $\ZZ[x]/(f(x))$ into the real vector space $\RR^n$ under the Minkowski embedding (before reduction modulo $q$), and uses a Gaussian in $\RR^n$; this induces an entirely different distribution on the error vectors for general number rings.  It was shown in~\cite{LPR10} and~\cite{DD12} that in the case of $2$-power cyclotomic rings, the distributions are the same.   However, in~\cite{ELOS} it was shown that in general rings the distortion of the distribution is governed by the largest singular value of the change-of-basis matrix between the Minkowski and the polynomial basis.  Thus the RLWE and PLWE distributions are not equivalent in general rings.

Although RLWE and PLWE for cyclotomic rings, particularly two-power cyclotomic rings, are the current candidates for practical lattice-based homomorphic encryption with ideal lattices, it will be important for a full study of their security to consider the RLWE and PLWE problems for general rings.  This includes studying the two problems independently, and analysing their relationships via the distortion of distributions just mentioned.  One lesson from \cite{ELOS} is that deviating from these recommended candidates can lead to an insecure system.

The RLWE and PLWE problems are formulated as either `search' or `decision' problems (see Section \ref{PLWE} below).  A security reduction was presented in~\cite{LPR10} showing that, for any {\it cyclotomic} ring $R$, an algorithm for the decision version of the Ring-LWE problem yields a comparably efficient algorithm for the search version of the Ring-LWE problem. This search-to-decision reduction was subsequently extended to apply to any Galois field in~\cite{EHL}.

In~\cite{EHL}, an attack on PLWE was presented in rings $P_q=\FF_q[x]/(f(x))$, where $f(1) \equiv 0 \pmod{q}$.  In addition, ~\cite{EHL} gives sufficient conditions on the ring so that the
`search-to-decision' reduction for RLWE holds, and also that RLWE instances can be translated into PLWE instances, so that the RLWE decision problem can be reduced to the PLWE decision problem.
Thus, if a number field $K$  satisfies the following six conditions simultaneously, then the results of ~\cite{EHL} give an attack on the search version of RLWE:
  \begin{enumerate}
        \item \label{galois} $K=\mathbb Q(\beta)$ is Galois of degree $n$.
        \item\label{split} The ideal $(q)$ splits completely in $R=\mathcal O_K$, the ring of integers of $K$,  and $q \nmid [R:\mathbb Z[\beta]]$.
        \item \label{mono} $K$ is monogenic, i.e., the ring of integers $R=\mathcal O_K$ of $K$ is generated by one element $R=\mathbb Z[\beta]$.
        \item \label{orthogonal} The transformation between the canonical embedding of $K$ and the power basis representation of $K$ is given by a scaled orthogonal matrix.
        \item \label{small} If $f$ is the minimal polynomial of $\beta$, then $f(1) \equiv 0 \; \pmod q$.
        \item \label{big} The prime $q$ can be chosen suitably large.
\end{enumerate}
The first two conditions are sufficient for the RLWE search-to-decision reduction; the next two conditions are sufficient for the RLWE-to-PLWE reduction; and the last two conditions are sufficient for the attack on PLWE.
Unfortunately, it is difficult to construct number fields satisfying all six conditions simultaneously.  In ~\cite{EHL} examples of number fields were given which are vulnerable to the attack on PLWE.

In~\cite{ELOS},  the attack on PLWE was extended by weakening the conditions on $f(x)$ and the reduction from RLWE to PLWE was extended by weakening condition (\ref{orthogonal}).  A large class of fields were constructed where the attack on PLWE holds and RLWE samples can be converted to PLWE samples, thus providing examples of weak instances for the RLWE problem.

Exciting number theory problems often arise from cryptographic applications.
In this paper we survey and extend the attacks on the PLWE and RLWE problems and raise associated number theoretic questions. In Section \ref{PLWE}, we recall the PLWE and RLWE problems. In Sections \ref{sum} and \ref{residue} we survey and extend the attacks on PLWE  which were introduced in~\cite{EHL,ELOS}. In Section~\ref{reduction}, we explain the reduction between the RLWE and PLWE problems.  Finally in Section \ref{NT} we raise related questions in number theory; in particular, we investigate the spectral distortion of an algebraic number and its relationship to Mahler measure, the monogenic property for the ring of integers of a number field, and the size of elements of small order modulo $q$.

\section{The fundamental hard problems:  PLWE and RLWE} \label{PLWE}

\subsection{PLWE}

Take $f(x) \in\ZZ[x]$ to be monic and irreducible of degree $n$.  Suppose $q$ is a prime modulo which $f(x)$ factors completely (this is not necessary for the definition of the problem, but we will assume this throughout the paper).  Write
\[
        P := \ZZ[x]/f(x), \quad P_q := P/qP \cong \FF_q[x]/f(x).
\]
Let $\sigma \in \RR^{>0}$.  By a \emph{Gaussian distribution $\mathcal{G}_{\sigma}$ of parameter $\sigma$},
we mean a Gaussian of mean $0$ and variance $\sigma^2$ on $P$ which is spherical with respect to the power basis $1, x, x^2, \ldots, x^{n-1}$ of $P$.  The prime $q$ is generally assumed to be polynomial in $n$, sometimes as large as $2^{50}$ but in some applications much smaller (even as small as $2^{12}$), and $\sigma$ is taken fairly small (perhaps $\sigma = 8$), so that in practice the tails of the Gaussian will decay to negligible size well before its variable reaches size $q$.  Since $P$ has integer coordinates, we must `discretize' the Gaussian in an appropriate fashion; the result is simply referred to as a \emph{discretized Gaussian}.  We will not go into the technical details in this paper, but instead refer the reader to \cite{LPR10}.

There are two standard PLWE problems, quoted here from \cite{BV11}.  The difficulty involves determining a secret obscured by a small \emph{error} drawn from the discretized Gaussian.

\begin{problem}[Search PLWE Problem] \label{PLWES}
         Let $s(x) \in P_q$ be a secret. The \emph{search PLWE problem}, is to discover $s(x)$ given access to arbitrarily many independent samples of the form $(a_i(x), b_i(x) := a_i(x)s(x) + e_i(x)) \in P_q \times P_q$, where for each $i$, $e_i(x)$ is drawn from a discretized Gaussian of parameter $\sigma$, and $a_i(x)$ is uniformly random.
\end{problem}

The polynomial $s(x)$ is the \emph{secret} and the polynomials $e_i(x)$ are the \emph{errors}.
There is a decisional version of this problem:

\begin{problem}[Decision PLWE Problem] \label{PLWED} Let $s(x) \in P_q$ be a secret.  The \emph{decision PLWE problem} is to distinguish, with non-negligible advantage, between the same number of independent samples in two distributions on $P_q \times P_q$.  The first consists of samples of the form $(a(x), b(x) := a(x) s(x) + e(x) )$ where $e(x)$ is drawn from a discretized Gaussian distribution of parameter $\sigma$, and $a(x)$ is uniformly random.  The second consists of uniformly random and independent samples from $P_q \times P_q$.
\end{problem}

Search-to-decision reductions were proved for cyclotomic number fields in~\cite{LPR10} and extended to work for Galois number fields in~\cite{EHL}.
Of course, the phrase `to distinguish' must be interpreted to mean that the distinguisher's acceptance probabilities, given PLWE samples versus uniform samples, differ by a non-negligible amount.

\subsection{RLWE} \label{RLWE}

The original formulation of the hard learning problem for rings, RLWE, presented in ~\cite{LPR10}, was based on the ring of integers, $R$,  of a number field.  The authors studied a general class of problems where the error distribution was allowed to vary.

Here we are concerned with only two choices of distributions.  The first is to consider rings, $R$, which are isomorphic to a polynomial ring $P$, and study the PLWE problem (PLWE was stated as a ``variant'' of RLWE in~\cite{LPR10} and ~\cite{BV11}).  The distribution in this case is with respect to the polynomial basis of one of its polynomial representations.

The second is to choose the error according to a discretized Gaussian with respect to a special basis of the ambient space in which $R$ was embedded via the Minkowski embedding.  We will refer to this as RLWE.  Therefore, in our language, when $R$ is isomorphic to some polynomial ring $P$, RLWE differs from PLWE only in the error distribution.

We will state the fundamental RLWE problems and then discuss the relationship between the RLWE and PLWE problems.  Let $K$ be number field of degree $n$ with ring of integers $R$. Let $R^\vee$ denote the dual of $R$,
\[
        R^\vee = \{ \alpha \in K: Tr(\alpha x) \in \ZZ \mbox{ for all } x \in R \}.
\]
Let us write
$R_q := R/qR$ and $R_q^\vee = R^\vee/qR^\vee$.  We will embed $K$ in $\CC^n$ via the usual Minkowski embedding.  The vector space $\CC^n$ is endowed with a standard inner product, and we will use the spherical Gaussian with respect to this inner product, discretized to $R^\vee$, as the discretized Gaussian distribution.  We will refer to this as the \emph{canonical discretized Gaussian}.  This will \emph{not}, in general, coincide with the discretized Gaussian defined in PLWE for a $P \cong R$, and this is the fundamental difference between the two problems.

The standard RLWE problems for a canonical discretized Gaussian are as follows.

\begin{problem}[Search RLWE Problem \cite{LPR10}]
        Let $s \in R_q^\vee$ be a secret.  The \emph{search RLWE problem} is to discover $s$ given access to arbitrarily many independent samples of the form $(a, b:=as +e )$ where $e$ is drawn from the canonical discretized Gaussian and $a$ is uniformly random.
\end{problem}

\begin{problem}[Decision RLWE Problem \cite{LPR10}]
        Let $s \in R_q^\vee$ be a secret.  The \emph{decision RLWE problem} is to distinguish with non-negligible advantage between the same number of independent samples in two distributions on $R_q \times R_q^\vee$.  The first consists of samples of the form $(a, b:=as +e )$ where $e$ is drawn from the canonical discretized Gaussian and $a$ is uniformly random, and the second consists of uniformly random and independent samples from $R_q \times R_q^\vee$.
\end{problem}

\bigskip

An isomorphism between $R$ and an appropriate polynomial ring $P$ can be used to relate an instance of the RLWE problem to an instance of the PLWE problem.  In particular, one requires $R$ to be monogenic (having a power basis).  Analysing the relationship between the two problems involves a close look at the change of basis under an isomorphism from $R$ to the appropriate $P$.  We will take up this issue in Section \ref{reduction}.

\section{Summary of Attacks }\label{sum}
\label{sec:attack}

In practice today, parameters for cryptosystems based on the RLWE and PLWE problems are set according to two known attacks, the {\it distinguishing attack} (\cite{MR,RS}) and the {\it decoding attack} (\cite{LP}).  These attacks work in general for learning-with-error problems and do not exploit the special structure of the ring versions of the problem.
In this paper, we will focus solely on the new attacks  presented in \cite{EHL} and \cite{ELOS} that exploit the special number-theoretic structure of the PLWE and RLWE rings.

The attacks presented in \cite{EHL} and \cite{ELOS} can be described in terms of the  ring homomorphisms from $P_q$ to smaller rings.  As $P_q \cong \FF_q^n$, the only candidates are the projections to each factor:
\[
        \pi_\alpha: P_q \rightarrow \FF_q, \quad p(x) \mapsto p(\alpha)
\]
for each  root $\alpha$ of $f(x)$.  In $P_q$, the short vectors sampled by the Gaussian are easy to recognise since they have small coefficients.  But they are hard to tease out of $b(x) = a(x)s(x) + e(x)$ without knowledge of $s(x)$, and the possibilities for $s(x)$ are too many to examine exhaustively.  By contrast, in a small ring like $\FF_q$, it is easy to examine the possibilities for $s(\alpha)$ exhaustively.  And the ring homomorphism preserves the relationship of the important players: $b(\alpha) = a(\alpha)s(\alpha) + e(\alpha)$.  Hence we can loop through the possibilities for $s(\alpha)$, obtaining for each the putative value
\[
        e(\alpha) = b(\alpha) - a(\alpha)s(\alpha).
\]
The Decision Problem for PLWE, then, is solved as soon as we can recognize the set of $e(\alpha)$ that arise from the Gaussian.

Unfortunately (or fortunately), one does not expect to be able to do this in general.  Heuristically, let $\mathcal{S} \subset P_q$ denote the subset of polynomials that are produced by the Gaussian with non-negligible probability.  In $P_q$, parameters are such that this is a small set.  But $\FF_q$ is a much smaller ring and one expects that generically, the image of $\mathcal{S}$ will `smear' across all of $\FF_q$.  Something quite special must happen if we expect the image of $\mathcal{S}$ to remain confined to a small subset of $\FF_q$, and hence be recognisable.

That `something special' is certainly possible, however:  suppose that $\alpha=1$.  The polynomials $g(x) \in \mathcal{S}$ have small coefficients, and hence have small images $g(1)$ in $\FF_q$.  This is simply because $n$ is much smaller than $q$, so that the sum of $n$ small coefficients is still small modulo $q$.
More generally, all of the attacks suggested in \cite{EHL} and \cite{ELOS} come down to considering $\alpha$ with certain advantageous properties, so that the image of $\mathcal{S}$ can be recognised.

The cyclotomic cases currently under consideration for PLWE and RLWE are uniquely protected against this occurrence:  $\alpha = 1$ is never a root modulo $q$ of a cyclotomic polynomial of degree $> 1$ when $q$ is sufficiently large.

\bigskip

\subsection{Attacking $\alpha=1$}\label{alpha1}

The approach described above and the $\alpha=1$ attack was first presented in \cite{EHL}. The details are as follows. Suppose $$f(1) \equiv 0 \hbox{ mod } q.$$  We are given access to a collection  of samples $(a_i(x), b_i(x))$.  We wish to determine if a sample is \emph{valid}, of the form $$b_i(x) = s(x)a_i(x) + e_i(x)$$ for $e_i(x)$ produced by a Gaussian, or \emph{random} (uniformly random).  The algorithm is as follows:

\noindent
{\bf Algorithm 1:}
\begin{enumerate}
        \item Let the set of valid guesses be $S = \FF_q$.
        \item Loop through the available samples.  For each sample:
                \begin{enumerate}
        \item Loop through guesses $s \in S$ for the value of $s(1)$.  For each $s$:
                \begin{enumerate}
        \item Compute $ e_i := b_i(1)-sa_i(1) $
        \item If $e_i$ is \emph{not} small in absolute value\footnote{meaning residue of smallest absolute value} modulo $q$, then conclude that the sample cannot be valid for $s$ with non-negligible probability, and remove $s$ from $S$.
        \end{enumerate}
        \end{enumerate}
\item If $S = \emptyset$, conclude that the sample was random.  If $S$ is non-empty, conclude that the sample is valid.
\end{enumerate}

        If the guess $s$ is correct, then $e_i = e_i(1)=\sum_{j=1}^n e_{ij}$
where $e_{ij}$ are chosen from a Gaussian $\mathcal{G}_{\sigma}$ of parameter $\sigma$.  It follows that $e_i(1)$ are approximately sampled from a Gaussian $\Gcal_{\sqrt{n}\sigma}$ of parameter $\sqrt{n}\sigma$ where $n\sigma^2 \ll q$.

\subsection{Attacking $\alpha$ of small order}\label{order}

The following attack described and developed in \cite{EHL, ELOS} requires $\alpha$ to have small order mod $q$.  The fundamental idea is the same as for the $\alpha=1$ attack, except that to discern whether or not $e_i(\alpha)$ is a possible image of a Gaussian-sampled error is more complicated.

Assume that $\alpha^r \equiv 1 \mod{q}$, then
\begin{equation*}
e(\alpha)= \sum_{i=1}^n e_i \alpha^i = (e_r+e_{2r}+ \cdots)+ \cdots + \alpha^{r-1}(e_{r-1}+e_{2r-1}+\cdots).
\end{equation*}
If $r$ is small enough, $e(\alpha)$ takes on only a small number of values modulo $q$.  This set of values may not be easy to describe, but $q$ is small enough that it can be enumerated and stored.  The attack proceeds as for $\alpha=1$ except that to determine if a sample is potentially valid for $s$ in step (2)(a)(ii), we compare to the stored list of possible values.

\section{Attacking $\alpha$ of small residue}\label{residue}

A third attack described in \cite{ELOS} is based on the size of the residue $e_i(\alpha)$ mod $q$.  This is more subtle.  Here, the errors $e(\alpha)$ may potentially take on all values modulo $\FF_q$ with non-negligible probability.  But it may be possible to notice if the probability distribution across $\FF_q$ is not uniform, given enough samples.

This method of attack differs from the previous ones, but is also applicable to $\alpha=1$ and $\alpha$ of small order, so all cases will be treated together.


Assume that
\begin{align} \label{root}
f(\alpha) \equiv 0 \mod q
\end{align} for some $\alpha$.
Let $E_i$ be the event that $$b_i(\alpha)-ga_i(\alpha) \mod q \hbox{ is in the interval } [-q/4, q/4)$$ for some sample $i$ and guess $g$ for $s(\alpha) \mod q$.
        We wish to compare the probabilities
        $$P(E_i \ | \ \mathcal{D} = \mathcal{U}) \ \ \mbox{and} \ \ P(E_i\ | \ \mathcal{D} = \mathcal{G}_{\sigma}).$$
Here, $\mathcal{D} = \mathcal{U}$ refers to the situation where $b_i$ is uniformily random, while $\mathcal{D} = \mathcal{G}_{\sigma}$ refers to the situation where $b_i$ is obtained as $a_i s +e_i$ for some secret $s$, where $e_i$ follows a Gaussian $\mathcal{G}_{\sigma}$ truncated at $2 \sigma$ (in practice, the Gaussian is truncated as the tails decay to negligible values).
If $\mathcal{D} = \mathcal{U}$, then $b_i(\alpha)-ga_i(\alpha)$ is random modulo $ q$ for all guesses $g $, that is,
$$P(E_i \ | \ \mathcal{D} = \mathcal{U}) =\dfrac{1}{2}.$$
If $\mathcal{D} = \mathcal{G}_{\sigma}$, then $b_i(\alpha)-s(\alpha)a_i(\alpha)=e_i(\alpha) \mod{q}$. Indeed, the terms of $b_i(\alpha)-s(\alpha) a_i(\alpha)$ that are a multiple of $f$ vanish at $\alpha$ modulo $q$ by Assumption \eqref{root}.
We consider $$e_i(\alpha)=\sum_{j=0}^{n-1} e_{ij} \alpha^j,$$ where $e_{ij}$ is chosen according to the distribution $\mathcal{G}_{\sigma}$ and distinguish three cases corresponding to
\begin{enumerate}
\item $\alpha = \pm 1$ \label{one}
\item $\alpha \neq \pm 1$ and $\alpha$ has small order $r$ modulo $ q$ \label{two}
\item $\alpha \neq \pm 1$ and $\alpha$ is not of small order $r$ modulo $ q$ \label{three}
\end{enumerate}
We will now drop the subscript $i$ for simplicity.  In Case \eqref{one}, the error $e( \alpha )$ is distributed according to $\mathcal{G}_{ \bar{\sigma}}$ where $$\bar{\sigma}=  \sigma \sqrt{n}.$$
In Case \eqref{two}, the error can be written as
$$e(\alpha)= \sum_{i=0}^{r-1} e_i \alpha^i = (e_{0}+e_{r}+ \cdots) + \alpha(e_1+e_{r+1}+\cdots) + \cdots + \alpha^{r-1}(e_{r-1}+e_{2r-1}+\cdots)$$
where we assume that $n$ is divisible by $r$ for simplicity.
For $j=0, \cdots,r -1,$ we have that $$e_j + e_{ j+r } + \cdots + e_{  j+ n-r }  $$ is distributed according to $ \mathcal{G}_{\tilde{\sigma} }$ where $$ \tilde{\sigma} = \sigma \sqrt{\frac{n}{r}}.$$
As a consequence $e(\alpha)$ is sampled from $\mathcal{G}_{ \bar{\sigma} }$ where $$\bar{\sigma}^2=\sum_{i=0}^{r-1} \tilde{\sigma}^2 \alpha^{2i} = \sum_{i=0}^{r-1} \dfrac{n}{r} \sigma^2 \alpha^{2i} = \dfrac{n}{r} \sigma^2 \dfrac{\alpha^{2r}-1}{\alpha^2-1}.$$
In Case \eqref{three}, the error $e( \alpha )$ is distributed according to $\mathcal{G}_{ \bar{\sigma}}$ where $$ \bar{\sigma}^2= \sum_{i=0}^{n-1} \sigma^2 \alpha^{2i}= \sigma^2 \dfrac{\alpha^{2n}-1}{\alpha^2-1}.$$
If $\dfrac{q}{4} \geq 2 \bar{\sigma}$, then errors always lie in $[-\frac{q}{4}, \frac{q}{4} )$ and
$$P(E_i\ | \ \mathcal{D} = \mathcal{G}_{\sigma}) = 1.$$
Otherwise, assuming for simplicity that $ N=\dfrac{2 \bar{\sigma}-q/2}{q} $ is an integer, we have
$$P(E_i\ | \ \mathcal{D} = \mathcal{G}_{\sigma}) =
 \left(\int_0^{ 2 \bar{\sigma}} \mathcal{G}_{\bar{\sigma}} \right)^{-1} \left(  \int_0^{\frac{q}{4}} \mathcal{G}_{\bar{\sigma}} +\sum_{k=0}^{N-1 }\int_{\frac{3q}{4}+kq}^{\frac{5q}{4}+kq} \mathcal{G}_{\bar{\sigma}} \right).$$
In the situation where this value exceeds $1/2$, i.e., $P(E_i\ | \ \mathcal{D} = \mathcal{G}_{\sigma})=\dfrac{1}{2}+\epsilon$ with $\epsilon>0$, the following algorithm attacks PLWE.
Let $$N= \left\lceil \dfrac{\ell q+\epsilon \ell}{2} \right\rceil$$ where $\ell$ is the number of samples observed.
For each guess $g$ mod $q$, we compute $b_i - g a_i \mod q $ for $i=1, \cdots , \ell$.
We denote by $C$ the number of elements obtained in the interval $[-q/4, q/4)$.
If $C<N$, the algorithm outputs $$\mathcal{D} = U,$$ otherwise, the algorithm outputs $$\mathcal{D} = \mathcal{G}_{\sigma}.$$

In the analysis of the probability of success of the algorithm, we denote by $B$ the binomial distribution and by $F$ the cumulative Binomial distribution.
If $\mathcal{D} = \mathcal{U}$, the algorithm is successful with probability
$$P(C<N | \mathcal{D} = U)=F(N-1; \ell q, \frac{1}{2}).$$
If $\mathcal{D} = \mathcal{G}_{\sigma}$,
we denote by $C_s$ the number of elements of the form $b_i -s a_i \mod q $ in the interval $[-q/4, q/4)$. In this case, the algorithm is successful with probability
\begin{align*}
& P\left(C \geq N|\mathcal{D} = \mathcal{G}_{\sigma} \right)=\sum_{i=0}^{\ell} P\left( C-C_s \geq N-i \right) \times P\left( C_s =i \right)\\
&=\sum_{i=0}^{\ell} \left(1-F(N-i-1,\ell q-\ell,1/2) \right) \times B(i, \ell , 1/2+ \epsilon) \\
\end{align*}
When $\epsilon>0$, the algorithm is successful since
\begin{align*}
& \dfrac{1}{2}(P(C<N | \mathcal{D} = U)+P(C \geq N|\mathcal{D} = \mathcal{G}_{\sigma} )) \\
& = \dfrac{1}{2}(P(C<N | \mathcal{D} = U)+1-P(C < N|\mathcal{D} = \mathcal{G}_{\sigma} )) \\
& = \dfrac{1}{2} + \dfrac{1}{2}(P(C<N | \mathcal{D} = U)-P(C < N|\mathcal{D} = \mathcal{G}_{\sigma} )) > \dfrac{1}{2}
\end{align*}

\begin{example} \label{Ex4.1}
In Case \eqref{one}, when $n=2^{10}$, $q \approx 2^{50}$, and $\sigma=8$, we can compute $\epsilon \approx 0.5$. Therefore, the attack is successful for any irreducible polynomial of degree $2^{10}$ and with a root $1$ mod $q$.
\\
In Case \eqref{two}, when $n=2^9$, $q\approx 2^{50}$, $\sigma=8$, and $\alpha=q-1$, $\alpha$ has order 2 and we can compute $\epsilon \approx 0.002$. This is particularly interesting since there is an irreducible polynomial with these properties that generates a power of 2 cyclotomic number field \cite{ELOS}; however, it is not the usual cyclotomic polynomial.
\\
In Case \eqref{three}, when $n=2^6$, $q\approx 2^{60}$, $\sigma=8$, and $\alpha=2$, computations show that $\epsilon=0.02$. Therefore, this attack is successful for any irreducible polynomial of degree $2^6$ with a root $\alpha = 2$ modulo a prime  $q \approx 2^{60}$.
\end{example}

\section{ RLWE-to-PLWE reduction} \label{reduction}
\label{sec:reduc}

Suppose that $K$ is a number field, and $R$ is its ring of integers.  For technical reasons, we give a slight variant on the Minkowski embedding, which is as follows:
$     \theta : K \rightarrow \RR^n $
\begin{align*}
\theta(r) :=  (\sigma_1(r), \ldots, \sigma_{s_1}(r), & Re(\sigma_{s_1+1}(r)), \ldots
 \ldots Re(\sigma_{s_1 + s_2}(r)),\\
& Im(\sigma_{s_1 + 1}(r)), \ldots, Im(\sigma_{s_1+s_2}(r)) ).
\end{align*}
where the $\sigma_i$ are the $s_1 + s_2$ embeddings of $K$, ordered so that the $s_1$ real embeddings are first, and the $s_2$ complex embeddings are paired so that $\overline{\sigma_{s_1+k}} = \sigma_{s_1 + s_2 + k}$.

A spherical Gaussian of parameter $\sigma$ with respect to the usual inner product on $\RR^n$ can be discretized to the \emph{canonical discretized Gaussian} on $R$ or its dual $R^\vee$.

Suppose $R \cong P$ for some polynomial ring $P$ under a map $\alpha \mapsto x$ for some root $\alpha$ of $f(x)$.  Suppose further that $R$ is monogenic.  Then $R^\vee \cong P$ also as $R$-modules (as its different ideal is principal).  For RLWE, $\RR\otimes R^\vee$ is equipped with a basis $\mathbf{b}_i$, $i=0, \ldots, n-1$ with respect to which the Gaussian is spherical (the standard basis of $\RR^n$, pulled back by $\theta$).  For PLWE, $\RR \otimes P$ is equipped with such a basis also, i.e., the standard power basis $x^{i}$, $i=0,\ldots, n-1$.  To relate the two problems, one must write down the change-of-basis matrix between them.  It is the matrix
\[
        N_\alpha := \gamma M_\alpha^{-1}: \RR \otimes R^\vee \rightarrow \RR \otimes P
\]
where $\gamma$ is such that $R^\vee = \gamma R$, and where $M_\alpha$ is the matrix with columns $[\alpha^i]_{\mathbf{b}}$ (i.e., the $i$-th column is the element $\alpha^i$ represented with respect to the basis $\mathbf{b} = \{ \mathbf{b}_i \}$).

The properties of $N_\alpha$ determine how much the Gaussian is distorted in moving from one problem to the other.  If it is not very distorted, then solving one problem may solve the other.

Details are to be found in \cite{ELOS}, but in short, the \emph{normalized spectral norm} gives a good measure of `distortion'.  This is defined for an $n\times n$ matrix $M$ by
\[
        ||M||_2/\operatorname{det}(M)^{1/n}.
\]


\section{Number Theoretical Open Problems} \label{NT}
In this section we will describe a number of open problems in number theory that are motivated by attacks to PLWE and RLWE, some very speculative and some more precise.

\subsection{Conditions for smearing}

As described in Section \ref{sec:attack}, we are concerned with the map
\[
        \pi: P_q \rightarrow \FF_q, \quad g(x) \mapsto g(\alpha).
\]

\begin{question}
        For which subsets $\mathcal{S} \subset P_q$, is the image $\pi(\mathcal{S}) = \FF_q$?
\end{question}

If $\pi(\mathcal{S}) = \FF_q$, we will say that $\mathcal{S}$ \emph{smears} under $\pi$.

Partial solutions to this problem may come in a wide variety of shapes.  For example, can one prove that almost all $\mathcal{S}$ of a given size smear?  Can one characterise the types of situations that lead to a negative answer (e.g. $\alpha=1$ and $\mathcal{S}$ consisting of polynomials of small coefficients)?  What if we restrict to the PLWE case, where $\mathcal{S}$ consists of polynomials with small coefficients?  Or the RLWE case, where $\mathcal{S}$ is the image of a canonical discretized Gaussian?

\bigskip

\subsection{The spectral distortion of algebraic numbers, and Mahler measure}

By Section \ref{sec:reduc}, the normalized spectral norm of $N_\alpha$ is a property of any algebraic number $\alpha$ for which $\ZZ[\alpha]$ is a maximal order.  We will therefore denote it $\rho_\alpha$, and call it the \emph{spectral distortion of $\alpha$}.  It measures the extent to which the power basis $\alpha^i$ is distorted from the canonical basis of the associated number field. Recall from Section~\ref{reduction} that for number rings with small spectral distortion we expect to have an equivalence between the RLWE and PLWE problems.  For completeness, we state a slightly more general definition, separate from its cryptographic origins, as follows:

\begin{definition}
  Let $\alpha$ be an algebraic number of degree $n$ and $K = \QQ(\alpha)$.  Let $M$ be the matrix whose columns are given by $\theta(\alpha^i)$, where  $ \theta: K \rightarrow \RR^n$,
\begin{align*}
\theta(r) = (\sigma_1(r), \ldots, \sigma_{s_1}(r), & Re(\sigma_{s_1+1}(r)),  \ldots Re(\sigma_{s_1 + s_2}(r)),\\
& Im(\sigma_{s_1 + 1}(r)), \ldots, Im(\sigma_{s_1+s_2}(r)) )
\end{align*}
where the $\sigma_i$ are the $s_1 + s_2$ complex embeddings of $K$, ordered so that the $s_1$ real embeddings are first, and the $s_2$ complex embeddings are paired so that $\overline{\sigma_{s_1+k}} = \sigma_{s_1 + s_2 + k}$.  The \emph{spectral distortion of $\alpha$} is $||M||_2 / (\operatorname{det}(M))^{\frac{1}{n}}$.
\end{definition}

\begin{question}
        What are possible spectral distortions of algebraic numbers?
\end{question}

It follows from the special properties of $2$-power roots of unity that they have spectral distortion equal to $1$.
However, even other roots of unity do not have spectral distortion equal to $1$ (and this is
what necessitates the more elaborate RLWE-to-PLWE reduction argument given in~\cite{DD12} for
cyclotomic rings which are not $2$-power cyclotomics).

Is the spectral distortion a continuum, or is the collection of values discrete in regions of $\RR$?  Does this relate to Mahler measure?

The Mahler measure of a polynomial can be defined as the product of the absolute values of the
roots which lie outside the unit circle in the complex plane, times the absolute value of the leading coefficient.  For a polynomial
$$f(x) = a(x-\alpha_1)(x-\alpha_2) \cdots (x-\alpha_n)$$ the Mahler measure is
$$M(f) := |a| \prod_{|\alpha_i| \ge 1} |\alpha_i|.$$
The Mahler measure of an algebraic number $\alpha$ is defined as the Mahler measure of its minimal polynomial.


Interestingly, polynomials which have small Mahler measure (all roots very close to $1$ in absolute value), seem to have small spectral distortion.  For example, consider ``Lehmer's polynomial'', the polynomial with the smallest known Mahler measure greater than $1$:
$$f(x)= x^{10}+x^9-x^7-x^6-x^5-x^4-x^3+x+1.$$
The Mahler measure is approximately $1.176$, and the spectral distortion for  its roots is
approximately $3.214$.  This spectral distortion is rather small, and compares favorably for example with the spectral distortion for $11^{th}$ roots of unity, which is approximately $2.942$.  Other
examples of polynomials with small Mahler measure also have small spectral distortion:
$f(x)=x^3-x+1$ has Mahler measure approximately $1.324$ and spectral distortion approximately $1.738$.

To explain the phenomenon observed for polynomials with small Mahler measure and
to relate the Mahler measure to the spectral norm, we need to have some estimate on the spectral norm in terms of the entries of the matrix.  The entries of the matrix $M$ are powers of the roots $\{ \alpha_j\}$ of the minimal polynomial.
When the Mahler measure is small, the entries of the matrix $M$ have absolute value close to $1$, since the absolute values of the roots are as close as possible to $1$.  To make the connection with the spectral norm more precise, \cite{Nikiforov}
gives an improvement on Schur's bound and expresses the bound on the largest singular value in terms of the entries of the matrix.  Thus we can use Schur's bound or this improvement to see that polynomials with small Mahler measure must have relatively small spectral norm.

It could also be interesting to look at other properties of $M$, such as the entire vector of singular values of $M$, its conditioning number, etc.

\bigskip

\subsection{Galois versus Monogenic}

 We say that $K$ is {\it monogenic} if the ring of integers $R$ of $K$ is monogenic, i.e.,  a simple ring extension $\mathbb Z[\beta]$ of $\mathbb Z.$ In this case, $K$ will have an integral basis of the form $\{1,\beta, \beta^2,\ldots, \beta^{n-1}\}$ which is called a \emph{power integral basis}.  In this section we will focus on properties (1) and (4) from the introduction.


\begin{example} The following are examples of number fields that are both Galois and monogenic:

\begin{itemize}
\item Cyclotomic number fields, $K=\mathbb Q(\zeta_n)$ where $\zeta_n$ is a primitive $n$th root of unity,
\item Maximal real subfields of cyclotomic fields, $K=\mathbb Q(\zeta_n+\zeta_N^{-1})$,
\item Quadratic number fields $K=\mathbb Q(\sqrt{d})$.
\end{itemize}
\end{example}

\begin{question}
Are there fields of cryptographic size which are Galois and monogenic, other than the cyclotomic number fields and their maximal real subfields? How can one construct such fields explicitly?
\end{question}

The problem of characterizing all  number fields which are monogenic goes back to Hasse, however,  a complete solution is not known to date.  Here we will summarize some of the known related results.

\begin{proposition} \cite{nakahara}
Let $p$ be a prime  and $K$  a Galois extension of $\mathbb Q$ of degree $n$. Let $e$ be the ramification index of $p$ and $f$ be the inertia degree of $p$. If one of the conditions below is satisfied then $K$ is not monogenic:
\begin{itemize}
\item  If $f=1$: $e p < n$
\item If $f \geq 2$: $e p^f \leq n+e-1$
\end{itemize}
\end{proposition}

Let $K$ be a Galois extension of prime degree $\ell$.  (Such extensions are called cyclic extensions.) The following result of Gras \cite{Gras} states that cyclic extensions are often non-monogenic.

\begin{theorem}\cite{Gras}
Any cyclic extension $K$ of prime degree $\ell \geq  5$ is non-monogenic except for the maximal real subfield of the  $(2\ell +1)$-th cyclotomic field with prime conductor $2\ell+1$.
\end{theorem}

\begin{theorem}\cite{Gras86}
Let $n \geq 5$ be relatively prime to $2,3$. There are only finitely
many abelian number fields of degree $n$ that are monogenic. \end{theorem}

For number fields of smaller degree it may be possible to give a complete characterization. For instance, for cyclic cubic extensions $K$, Gras \cite{Gras73} and Archinard \cite{Ar} gave  necessary and sufficient conditions for $K$ to be monogenic.

Even though monogenic fields are rare in the abelian case for large degree, Dummit and Kisilevsky \cite{DK} have shown that
there exist infinitely many cyclic cubic fields which are monogenic. A result of Kedlaya~\cite{K} implies that there are infinitely many monogenic number fields  of any given signature. In fact, we expect monogenicity frequently:  if $f$ is an irreducible polynomial with squarefree discriminant then the number field $K$ obtained by adjoining a root of $f$ to $\mathbb Q$ is monogenic.  For polynomials of fixed degree $\ge 4$ whose coefficients are chosen randomly, it is conjectured that with probability $\gtrsim 0.307$, the root will generate the ring of integers of the associated number field \cite{K}.  However, to require $K$ also to be Galois is much more restrictive. Moreover, for fields of cryptographic size ($n \sim 2^{10}$), the discriminant of $f$ is too large to test whether it is squarefree. Therefore testing whether an arbitrary number field of cryptographic size is monogenic is not known to be feasible in general.


\subsection{Finding roots of small order mod $p$}

We have seen that a root of small order of $f(x)$ modulo $q$ provides a method of attack on the PLWE problem in the ring $\ZZ_q[x]/(f(x))$.  The attack is even more effective if, in addition, this root is small as a minimal residue modulo $q$ (`minimal' meaning the smallest in absolute value). See Example~\ref{Ex4.1}, Case(3) in Section~\ref{residue}. Cyclotomic fields are protected against this attack by the observation that the roots of a cyclotomic polynomial modulo $q$ are of full order $n$.  However for `random' polynomials, there is a priori no particular reason to expect roots of any particular order modulo $q$, or to expect the roots to be small.  Motivated by these two requirements, it is natural to ask the following question:

\begin{question}
For random polynomials $f(x)$ and random primes $q$ for which $f(x)$ has a root $\alpha$ modulo $q$, what can one say about the order of $\alpha$ modulo $q$?
\end{question}

A special case of this question, for $f$ monic of degree one, is to ask, for a fixed $a$, how often is $a$ a primitive root modulo $p$?  A famous conjecture of Artin states that this should happen for infinitely many $p$ provided $a$ is not a perfect square or $-1$, and describes the density of such primes.  This has been the subject of much research, and the question above is a sort of number field analogue.  Some investigations in the direction of a number field analogue of Artin's conjecture exist; for a gateway to the literature, see \cite{M-P, Samuel}.

Computationally, to locate polynomials having a small root of small order, it is easiest to start with the desired order, find a suitable $q$, and then build the polynomial.  The algorithm is as follows:

\vspace{0.1in}
\noindent
{\bf Algorithm 2}

{\bf Input:} Integers $r, n, q_0$ such that $r>2$ represents the desired order, $n\ge1$ represents the desired degree, and $q_0>\log_2(n)$ represents the desired bitsize of $q$.
\begin{enumerate}
        \item Let $s$ be the degree of the cyclotomic polynomial $\Phi_r(x)$.
        \item Let $a = 1$ (our candidate for the element of order $r$ mod $q$).  Test $\Phi_r(a)$ for primality.  If it is a prime of approximate bitsize $q_0$, let $q$ be this prime.  Otherwise, increment $a$ and try again.
        \item Once $a$ and $q$ are fixed, choose a set $S$ of $n$ elements of $\ZZ/q\ZZ$ that includes $a$ and the other $n-1$ smallest minimal residues (or choose any other subset of residues).
        \item Choose $i=1$ and increment $i$ until the polynomial
                \[
                        f(x) = \prod_{s \in S}(x-s)+qi
                \]
              of degree $n$ is irreducible.
\end{enumerate}
{\bf Output:} A monic irreducible polynomial $f(x) \in \ZZ[x]$, a prime $q$ roughly of size $q_0$, such that $f$ splits modulo $q$, and $a \in \ZZ/q\ZZ$ such that $f(a)\equiv 0 \pmod q$ and $a^r \equiv 1 \pmod q$.

Note that if one wishes to relax the condition that $f$ splits modulo $q$, one could take $f(x) = (x-a)^n + q$, which is irreducible, to avoid Step 4.

\vspace{0.1in}

Using this method, it is easy to find examples of $(K,q)$ such that $f(x)$ has a root of small order modulo $q$.  Among them, an example of cryptographic size is afforded by $n = 2^{10}$, $r=3$, $a = 33554450$, $q = 1125901148356951$ and $i=1$ (the polynomial is too unwieldy to print here).  Using the last two parts of the method, one can, in fact, easily construct polynomials having as roots many elements of small order modulo $q$.

A simpler starting point is the following second question:

\begin{question}
        What is the distribution of elements of small order among residues modulo $q$?
\end{question}

There is a significant body of research on the distribution of primitive roots (see Artin's conjecture) and quadratic residues.  More recently there have been advances on the distribution of elements of small order.  For example, the number of elements of bounded size and specified order is bounded above in \cite{BKS2}; see also \cite{Bour,BKS,KS}.  More useful in our present context, for the purposes of finding elements of small order, would be a guarantee that such elements exist in some small interval.

A more precise question is as follows:

\begin{question}  What is the smallest residue modulo a prime $q$ which has order exactly $r$ ?
\end{question}

Let $q$ be a prime and $r>2$.  Let $n_{r,q}$ represent the smallest residue modulo $q$ which has order exactly $r$.  A first observation is the following (which allows us to choose a more suitable starting point for $a$ in the algorithm above).

\begin{proposition}
        \label{prop:nrq}
                Let $\varphi(r)$ represent the Euler function, giving the number of positive integers less than and coprime to $r$.  Then, if $r$ has at most two distinct prime factors, which are odd, then
                        \[
                                |n_{r,q}| \ge (q/\varphi(r))^{1/\varphi(r)}
                        \]
\end{proposition}

\begin{proof}
        The element $n_{r,q}$ is a root of the $r$-th cyclotomic polynomial, of degree $\varphi(r)$, modulo $q$.  Since $\Phi_r(n_{r,q}) \neq 0$ as an integer relation, it must be that $|\Phi_r(n_{r,q})| \ge q$.  It is known that under the given hypotheses on the factorisation of $r$, the coefficients of $\Phi_r$ are chosen from $\{ \pm 1, 0\}$ (\cite{Migotti}).  Therefore $|\Phi_r(n_{r,q})| \le \varphi(r)|n_{r,q}^{\varphi(r)}|$ from which the result follows.
\end{proof}

In general, combining upper and lower bounds on $n_{r,q}$ would limit the search space for an element of small order.

\begin{remark}

        \begin{enumerate}
                \item Other restrictions on the coefficients of $\Phi_r$ give rise to similar results.  To derive an asymptotic statement, one could turn to asymptotic results such as \cite{Erdos}.
                \item The case of $r=3$, the study of $n_{3,q}$ gives the full story, as the cube roots of unity are of the form
\[
        1, n_{3,q}, -n_{3,q}-1.
\]
\item In general, the primes $q$ such that $n_{r,q} = a$ for a fixed $a$ and $r$ are among those dividing $\Phi_r(a)$, hence there are finitely many.
\item Elliott has some results on $k$-th power residues \cite{Elliott}.
\end{enumerate}
\end{remark}

We will call $n_{r,q}$ \emph{minimal} if, in addition to being the smallest residue of order $r$ modulo $q$, it also satisfies $\Phi_r(n_{r,q}) = \pm q$.  For non-minimal $n_{r,q}$, the lower bound in Proposition \ref{prop:nrq} increases.  A conjecture of Bouniakowski implies that minimality happens infinitely often.

\begin{conjecture}[Bouniakowski, \cite{Boun}]
        Let $f(x) \in \ZZ[x]$ be a non-constant irreducible polynomial such that $f(x)$ is not identically zero modulo any prime $p$.  Then $f(n)$ is prime for infinitely many $n \in \ZZ$.
\end{conjecture}

\begin{proposition}
        \label{prop:minroot}
        Let $r >2$.  If Bouniakowski's Conjecture holds, then there are infinitely many primes $q$ for which $n_{r,q}$ is minimal.
\end{proposition}

\begin{proof}[Proof of Prop \ref{prop:minroot}]
        The cyclotomic polynomials for $r > 1$ satisfy the Bouniakowski conditions, as they are irreducible and $\Phi_r(1) \not\equiv 0 \pmod p$ since $1$ is not of exact order $r$ modulo any $p$.  Hence $\Phi_r(x)$ takes on infinitely many prime values; for such a prime $q$, the smallest such $x$ in absolute value is $n_{r,q}$ and this is minimal.
\end{proof}

{\bf Acknowledgements.}  The authors thank the organizers of the research conference Women in Numbers 3 (Rachel Pries, Ling Long and the fourth author) and the Banff International Research Station, for making this collaboration possible.
The authors also thank the anonymous referee for detailed comments and suggestions to improve the paper, and Igor Shparlinski for useful feedback and references.

\end{document}